\documentclass[12pt,reqno]{amsart}
\usepackage{a4}
\usepackage{amsmath, amssymb, amsthm}
\usepackage{graphicx, color}
\usepackage{enumerate}

\begin{document}

 \bibliographystyle{plain}
 \newtheorem{theorem}{Theorem}
 \newtheorem{lemma}[theorem]{Lemma}
 \newtheorem{corollary}[theorem]{Corollary}
 \newtheorem{problem}[theorem]{Problem}
 \newtheorem{conjecture}[theorem]{Conjecture}
 \newtheorem{definition}[theorem]{Definition}
 \newtheorem{prop}[theorem]{Proposition}
 \numberwithin{equation}{section}
 \numberwithin{theorem}{section}

 \newcommand{\mo}{~\mathrm{mod}~}
 \newcommand{\mc}{\mathcal}
 \newcommand{\rar}{\rightarrow}
 \newcommand{\Rar}{\Rightarrow}
 \newcommand{\lar}{\leftarrow}
 \newcommand{\lrar}{\leftrightarrow}
 \newcommand{\Lrar}{\Leftrightarrow}
 \newcommand{\zpz}{\mathbb{Z}/p\mathbb{Z}}
 \newcommand{\mbb}{\mathbb}
 \newcommand{\B}{\mc{B}}
 \newcommand{\cc}{\mc{C}}
 \newcommand{\D}{\mc{D}}
 \newcommand{\E}{\mc{E}}
 \newcommand{\F}{\mathbb{F}}
 \newcommand{\G}{\mc{G}}
  \newcommand{\ZG}{\Z (G)}
 \newcommand{\FN}{\F_n}
 \newcommand{\I}{\mc{I}}
 \newcommand{\J}{\mc{J}}
 \newcommand{\M}{\mc{M}}
 \newcommand{\nn}{\mc{N}}
 \newcommand{\qq}{\mc{Q}}
 \newcommand{\PP}{\mc{P}}
 \newcommand{\U}{\mc{U}}
 \newcommand{\X}{\mc{X}}
 \newcommand{\Y}{\mc{Y}}
 \newcommand{\itQ}{\mc{Q}}
 \newcommand{\sgn}{\mathrm{sgn}}
 \newcommand{\C}{\mathbb{C}}
 \newcommand{\R}{\mathbb{R}}
 \newcommand{\T}{\mathbb{T}}
 \newcommand{\N}{\mathbb{N}}
 \newcommand{\Q}{\mathbb{Q}}
 \newcommand{\Z}{\mathbb{Z}}
 \newcommand{\A}{\mathcal{A}}
 \newcommand{\ff}{\mathfrak F}
 \newcommand{\fb}{f_{\beta}}
 \newcommand{\fg}{f_{\gamma}}
 \newcommand{\gb}{g_{\beta}}
 \newcommand{\vphi}{\varphi}
 \newcommand{\whXq}{\widehat{X}_q(0)}
 \newcommand{\Xnn}{g_{n,N}}
 \newcommand{\lf}{\left\lfloor}
 \newcommand{\rf}{\right\rfloor}
 \newcommand{\lQx}{L_Q(x)}
 \newcommand{\lQQ}{\frac{\lQx}{Q}}
 \newcommand{\rQx}{R_Q(x)}
 \newcommand{\rQQ}{\frac{\rQx}{Q}}
 \newcommand{\elQ}{\ell_Q(\alpha )}
 \newcommand{\oa}{\overline{a}}
 \newcommand{\oI}{\overline{I}}
 \newcommand{\dx}{\text{\rm d}x}
 \newcommand{\dy}{\text{\rm d}y}
\newcommand{\cal}[1]{\mathcal{#1}}
\newcommand{\cH}{{\cal H}}
\newcommand{\diam}{\operatorname{diam}}
\newcommand{\bx}{\mathbf{x}}
\newcommand{\Ps}{\varphi}

\parskip=0.5ex

\title[Bounded remainder sets and cut-and-project sets]{Constructing bounded remainder sets and cut-and-project sets which are bounded distance to lattices}
\author{Alan~Haynes and Henna Koivusalo}
\thanks{Research supported by EPSRC grants EP/J00149X/1 and EP/L001462/1.}
\address{Department of Mathematics, University of York, York, UK}
\email{alan.haynes@york.ac.uk, henna.koivusalo@york.ac.uk}

\allowdisplaybreaks

\begin{abstract}
For any irrational cut-and-project setup, we demonstrate a natural infinite family of windows which gives rise to separated nets that are each bounded distance to a lattice. Our proof provides a new construction, using a sufficient condition of Rauzy, of an infinite family of non-trivial bounded remainder sets for any totally irrational toral rotation in any dimension.
\end{abstract}

\maketitle

\section{Introduction}
In this paper we consider the following two problems:
\begin{enumerate}
  \item[(i)] Given a cut-and-project setup corresponding to an $\R^d$ action on $\R^k$, construct a (non-trivial) infinite collection of $(k-d)$-dimensional windows which each give rise to a cut-and-project set that is bounded distance from a lattice.
  \item[(ii)] Given a minimal rotation of $\T^s:=\R^s/\Z^s$, explicitly construct a (non-trivial) infinite collection of regions in $\T^s$, for which the differences between the time and space averages of the system over each of the regions achieve the minimum possible asymptotic bound.
\end{enumerate}
As we will see in what follows, problems (i) and (ii) are closely related. A region which satisfies the condition in problem (ii) is called a {\em bounded remainder set} (BRS) for the corresponding rotation. For $s=1$ the problem is satisfactorily dealt with by works of Hecke \cite{Heck1922}, Ostrowski \cite{Ostr1927/30}, and Kesten \cite{Kest1966/67} (and also related results of Oren \cite{Oren1982}), which together show that for an irrational rotation of $\T$ by $\alpha$, the necessary and sufficient condition for an interval $\mc{I}$ to be a BRS is that $|\mc{I}|\in\alpha\Z+\Z.$ This translates easily into a solution of problem (i) for the special case when $k=d-1$. The first author has recently explored and improved upon this connection in \cite{Hayn2013}, using detailed arguments involving the theory of continued fractions. Several papers \cite{Liar1987,Szus1954,Zhur2005,Zhur2011,Zhur2012} have investigated problem (ii) when $s>1$. Of particular note are the works of Sz\"{u}sz \cite{Szus1954}, who demonstrated a construction of parallelogram BRS's when $s=2$, and Liardet \cite[Theorem 4]{Liar1987}, who used a dynamical cocycles argument to extend Sz\"{u}sz's result to arbitrary $s>1$.

Our goal in this paper is to provide a new constructive solution of both problems above, which works for all choices of dimensions and parameters involved. With respect to problem (ii), the collection of BRS's which we construct is different than the collection which can be constructed using Liardet's above mentioned result.  Our proofs rely on a beautiful result of Rauzy \cite{Rauz1972}, which provides a sufficient condition for a set in the $s$-torus to be a BRS.

To understand the connection between the above problems, first note that taking $d=1$ in problem (i) essentially gives a reformulation of problem (ii) (i.e. with $s=k-1$). Therefore solutions to problem (i) when $d=1$ immediately give solutions to problem (ii), for the corresponding toral rotations. It will become clear in the course of our proofs that this implication also goes the other way, so that solving problem (ii) gives solutions to problem (i). In brief, by using a BRS associated to just one direction in our $d$-dimensional physical space, we obtain enough structure to ensure that we can move the entire cut and project set to a lattice, moving each point by at most a bounded amount.

To state our results more precisely, we first make some definitions. Let $V$ be a $d$-dimensional subspace of $\R^k$, and let $\pi:\R^k\rar\T^k$ be the canonical projection. Suppose that $\mc{S}\subseteq\T^k$ is the image under $\pi$ of a bounded subset of a $(k-d)$-dimensional plane in $\R^k$ which is everywhere transverse to $V$ ($\mc{S}$ is what we will call a {\em section} or a {\em window}), and for each $x\in\R^k$ define $Y=Y_{\mc{S},x}\subseteq V$ by
\[Y_{\mc{S},x}=\{v\in V: \pi (v+x)\in \mc{S}\}.\]
We refer to $Y_{\mc{S},x}$ as the {\em cut-and-project set} associated to $k,V,\mc{S},$ and $x$. In much of what follows we will assume that $V$ is a {\em totally irrational} subspace of $\R^k$ (equivalently, that $\pi (V)$ is dense in $\T^k$). To see that this incurs no loss of generality, notice that every subspace of $\R^k$ has a dense orbit in some `sub-torus' of $\T^k$, so by re-parameterizing when necessary our problems can always be brought into such a situation.

Generically, a cut-and-project set as above is an aperiodic separated net (i.e. Delone set) in $V$. Several authors \cite{BuraKlei1998,BuraKlei2002,Hayn2013,HaynKellWeis2013,McMu1998,Solo2011} have recently addressed the question of determining how far away such sets can be from lattices in $V$. One way of measuring this is to say that two sets in $\R^k$ are {\em bounded distance} (BD) to one another if there is a bijection between them which moves each point by at most some fixed constant amount. This agrees exactly with the definition in other papers of {\em bounded displacement equivalence}. It was proved in \cite{HaynKellWeis2013} that for any $k>d\ge 2,$ for almost every $d$-dimensional subspace $V$ of $\R^k$ (in the sense of the natural measure on the Grassmannian manifold), and for every section $\mc{S}$ which is a $(k-d)$-dimensional aligned box (a box with all sides parallel to coordinate planes in $\R^k$), the set $Y$ is BD to a lattice in $V$. On the other hand it was also shown in the same paper that for almost every section $\mc{S}$ which is a parallelotope (or ball, ellipsoid, or suitably nice shape), there is a residual set of subspaces $V$ for which the corresponding sets $Y$ are not BD to any lattice. This shows that the problem of determining which cut-and-project sets are BD to a lattice is not trivial. It is also worth pointing out that the above mentioned results all rely on the Diophantine approximation properties of the subspace $V$. For subspaces $V$ which are extremely well approximable by rational subspaces, the above results give no information about cut-and-project sets coming from $V$. With this as a backdrop, we present our first result.
\begin{theorem}\label{thm.main}
For any $d,k,$ and $V$ as above, there is an infinite collection of $(k-d)$-dimensional sections with the property that, for any section $\mc{S}$ from the collection, and for any $x\in\R^k,$ the set $Y_{\mc{S},x}$ is BD to a lattice.
\end{theorem}
In our proof of this theorem we will construct an infinite family of sections (which necessarily depends on $V$) which satisfy the conclusion. Our aim is in fact to prove that these sections are examples of bounded remainder sets, of which Theorem \ref{thm.main} is a corollary. There is a related construction, due to Duneau and Oguey \cite[Theorem 3.1]{DuneOgue90}, which can be used to give an alternative proof of this theorem, however it does not illustrate the connection to the bounded remainder set problem.

Let us now formally define bounded remainder sets. For a Lebesgue measure preserving transformation $T:\T^s\rar\T^s$, we will say that a measurable set $A\subseteq\T^s$ is a {\it bounded remainder set} for $T$ if
\[
\sup_{x\in\R^s}\sup_{N\in\N}\left|\sum_{n=0}^{N-1}\chi_A(T^n(x)) - N|A|\right|<\infty,
\]
where $\chi_A$ is the indicator function of $A$ and $|A|$ denotes its measure. We are interested primarily in the case when $T$ is a rotation, given by $T(x)=x+\alpha,$ for some $\alpha\in\R^s$. Therefore, we may refer to a BRS for $T$ as a BRS for $\alpha$, or simply a BRS, if the context is clear. As in our discussion above, there is little loss of generality in assuming that $\alpha$ is a totally irrational rotation (i.e. that $\{n\alpha\}_{n\in\N}$ is dense in $\T^s$).

To fit the definition of bounded remainder sets into a larger framework, we digress for a moment. It is important in many contexts to quantify how evenly distributed a sequence of numbers is, modulo $1$. One way of doing this is to define, for $N\in\N$, the {\em discrepancy} $D_N$ of a sequence $\{x_n\}_{n=1}^\infty\subseteq\T^s$ by
\[D_N(\{x_n\})=\sup_{B\in\mc{B}}\left|\sum_{n=1}^N\chi_B(x_n)-N|B|\right|,\]
where the supremum is taken over all sets in some suitable family $\mc{B}$ of subsets of $\T^s$. It is common to take $\mc{B}$ to be the collection of all aligned boxes, which then leads to many important applications, for example the Koksma-Hlawka Inequality in numerical integration \cite[Theorem 1.14]{DrmoTich1997}. For aligned boxes (and some other classes of special shapes) a useful estimate for $D_N$ can be obtained by using the Erd\H{o}s-Tur\'{a}n-Koksma Inequality \cite[Theorem 1.21]{DrmoTich1997}. Bounding the discrepancy in this way involves estimating certain exponential sums. For the case when $x_n=n\alpha$ these sums grow large when $\alpha$ is well approximable by rational numbers, which is a limitation in many problems. Unfortunately the discrepancy estimates which are obtained in this way are not far from the truth. Results of van Aardenne-Ehrenfest \cite{vanA1945,vanA1949}, and later of Roth \cite{Roth1954} and W.~M.~Schmidt \cite{Schm1972} imply that, for any sequence, the discrepancy can never remain bounded as $N\rar\infty$. This highlights the special place in this theory occupied by bounded remainder sets.

In addition to the aforementioned work, bounded remainder sets have been studied by a number of authors. W. M. Schmidt \cite{Schm1974} showed that, for any sequence $\{x_n\}_{n=0}^\infty\subseteq\T^s$, there are at most a countably infinite number of possible volumes $|A|$ for measurable sets $A\subseteq\T^s$ satisfying
\[\sup_{N\in\N}\left|\sum_{n=0}^{N-1}\chi_A(x_n) - N|A|\right|<\infty.\]
For the special case when $x_n=x+n\alpha$, for some $x\in\R^s$ and totally irrational $\alpha\in\R^s,$  Liardet \cite{Liar1987} proved that the only examples of aligned boxes $A$ which are BRS's are the trivial ones which are derived from the $s=1$ problem. In other words, they consist only of products of intervals of the form
\[A=\prod_{r=1}^s \mc{I}_r,\]
for which there exists an $r'$ such that $|\mc{I}_{r'}|\in\alpha_{r'}\Z+\Z$, and $|\mc{I}_r|=1$ for all $r\not=r'$. These examples constitute what we refer to as the `trivial' solutions to problems (i) and (ii) above.

Non-trivial examples of BRS's for $s=2$ were given in \cite{Chev2009,Szus1954}, and for $s\ge 2$ in \cite{Liar1987,Zhur2005,Zhur2011,Zhur2012} (see also \cite{GrabHellLiar2012} for more discussion of what is known). As indicated above, combining Sz\"{u}sz and Liaret's results (\cite{Szus1954} and \cite[Theorem 4]{Liar1987}), one can obtain a nice algorithm for constructing examples of parallelotope BRS's in any dimension.

In this paper we are going to provide a simple construction, using a sufficient condition due to Rauzy \cite{Rauz1972}, which produces infinitely many examples of BRS's, for any $s$ and for any irrational rotation of $\T^s$. The BRS's which we construct in this way are what we will call {\em special regions}, and they are obtained by projections to the torus of parallelotopes coming from a lattice in $\R^{s+1}$ defined by the rotation.
\begin{theorem}\label{thm.mainBRS}
For any totally irrational rotation $\alpha$ of $\T^s$, every special region for $\alpha$ is a BRS.
\end{theorem}
For comparison with Sz\"{u}sz and Liardet's BRS's, we note that the parallelotopes obtained from Liardet's algorithm always have a face parallel to one of the coordinate hyperplanes in $\R^s$. Our special regions, on the other hand, typically do not have this property. It would be nice to understand the exact intersection of the two collections of regions constructed by our different algorithms, but this seems to be a technically difficult problem.

We will define special regions in Section \ref{sec.DescrSpecReg}, and show using Rauzy's criteria that they are BRS's. Then we will explain a method for explicitly constructing an infinite collection of special regions, for any irrational rotation. In Section \ref{sec.SpecRegToSpecSec} we will conclude by completing the proof of Theorem \ref{thm.main}.

{\em Acknowledgments:} We would like to thank John Hunton for inspiring us to think about projections of higher dimensional lattices, and Robert Tichy and Barak Weiss for pointing out several valuable references. We would also like to mention that Michael Kelly and Lorenzo Sadun have informed us that they have obtained results similar to our Theorem \ref{thm.main}, which should be forthcoming. Finally, we would like to thank the referee for his or her suggestions, which helped us to make several significant improvements in the paper.

\section{Special regions and Rauzy's criteria}\label{sec.DescrSpecReg}
Suppose that $\alpha\in\mathbb T^s$ is totally irrational and let $T:\T^s\rar\T^s$ be defined by $T(x)=x+\alpha$. Let $\alpha':=(\alpha,1)\in\R^{s+1}$, and define $\Lambda$ to be the lattice in $\R^{s+1}$ generated by $\alpha'$ together with the first $s$ standard basis vectors, $e_1,\ldots ,e_s$. Note that $\Lambda$ consists precisely of all vectors of the form
\[(n\alpha_1+a_1, \cdots, n\alpha_s+a_s,n),\]
where $n,a_1,\ldots ,a_s\in\Z$. For $1\le i_1<\cdots <i_j\le s+1$, we define $\Ps_{i_1\dots i_j}$ to be the projection map from $\R^{s+1}$ to the $(e_{i_1},\dots ,e_{i_j})$-plane in $\mathbb R^{s+1}$, and for simplicity of notation we set $\Ps:=\Ps_{1\dots s}$ and identify $\R^s$ with $\Ps(\R^{s+1})$. Finally, for any linearly independent vectors $u_1,\dots, u_m$ in $\mathbb R^{s+1}$, we denote by $P(u_1,\dots, u_m)$ the parallelotope which the vectors generate, in the linear subspace they span. We assume that the boundary of the parallelotope is chosen so that $P(u_1,\dots ,u_m)$ is a fundamental domain for the lattice generated by $u_1,\dots, u_m$, in the subspace which they span.

We say that a set $A$ in $\R^s$ is a {\em special region} for $\alpha$ if there exist vectors $v_1,\dots,v_{s+1}\in\Lambda$ satisfying the following conditions:
\begin{enumerate}[(S1)]
\item $A=\Ps(P(v_1, \dots, v_{s})),$
\item\label{cond.proj} $\Ps(v_{s+1})\in A$,
\item $v_1,\dots,v_{s+1}$ form a $\Z$-basis for $\Lambda$, and
\item\label{cond.up} For any subset $I\subset\{1,\dots, s\}$ we have that
\[\Ps_{s+1}(v_{s+1}) - \sum_{i\in I}\Ps_{s+1}(v_i)>0.\]
\end{enumerate}
We will explain how to construct such regions at the end of this section. First we focus on the proof of Theorem \ref{thm.mainBRS}.

Suppose that $T$ is a totally irrational rotation as above and that $A\subseteq\T^s$ is any set. Let $S:A\rar A$ be the map induced by $T$ on $A$ (i.e. the first return map to $A$). In \cite{Rauz1972} it is shown that for $A$ to be a BRS for $T$, it is sufficient that there exists a lattice $M\subseteq\R^s$ and a point $\beta\in\R^s$, such that
\begin{enumerate}[(R1)]
\item\label{cond.simple} If $a,b\in A$ satisfy $a=b\mo M$, then $a=b$, and
\item\label{cond.behaveModM} For all $x\in A$, we have that $S(x)= x+\beta\mo M$.
\end{enumerate}
Now we will show that every special region for $\alpha$ satisfies these conditions.

\begin{proof}[Proof of Theorem \ref{thm.mainBRS}]
Suppose that $A$ satisfies conditions (S1)-(S4), define $M$ to be the lattice in $\R^s$ generated by $\Ps(v_1), \dots, \Ps(v_{s})$, and let $\beta = \Ps(v_{s+1})$.

It is obvious from the definitions that condition (R1) is satisfied. In order to check (R2) we begin with some observations. First notice that the (forward and backward) orbit of $0$ under $S$ is encoded in the points of $\Lambda$ which lie in the cylinder $\Ps^{-1}(A)$. The last coordinate of each such point encodes its return time, with respect to $T$. In other words,
\[
\Lambda\cap\Ps^{-1}(A)= \left\{(S^n(0), \ell_n) : n\in\mathbb Z,~S^n(0)=T^{\ell_n}(0)\right\}.
\]
Next, denote by $H_0$ the hyperplane in $\R^{s+1}$ spanned by $v_1,\dots,v_s$, and for each $k\in\Z$ define $H_k:=H_0+kv_{s+1}.$ It follows from (S3) that the lattice $\Lambda$ can be written as the disjoint union
\[\Lambda=\bigcup_{k=-\infty}^\infty(\Lambda\cap H_k).\]
For each $k$, the set $(H_k\cap\Ps^{-1}(A))-kv_{s+1}$ is a fundamental domain for $H_0/(H_0\cap\Lambda)$. Therefore each set $(\Lambda\cap H_k)\cap\Ps^{-1}(A)$ contains exactly one point.



\begin{figure}\label{fig.windmills}
\centering
\def\svgwidth{0.65\columnwidth}
\caption{The red dots show the orbit of $0$ as it moves up through the hyperplanes $H_k$. When the point leaves the cylinder above $A$, it moves to a neighboring copy of $A$ in $H_k$, and must be translated back by an appropriate element of $M$.}
\vspace*{.2in}
\includegraphics[width=0.65\textwidth]{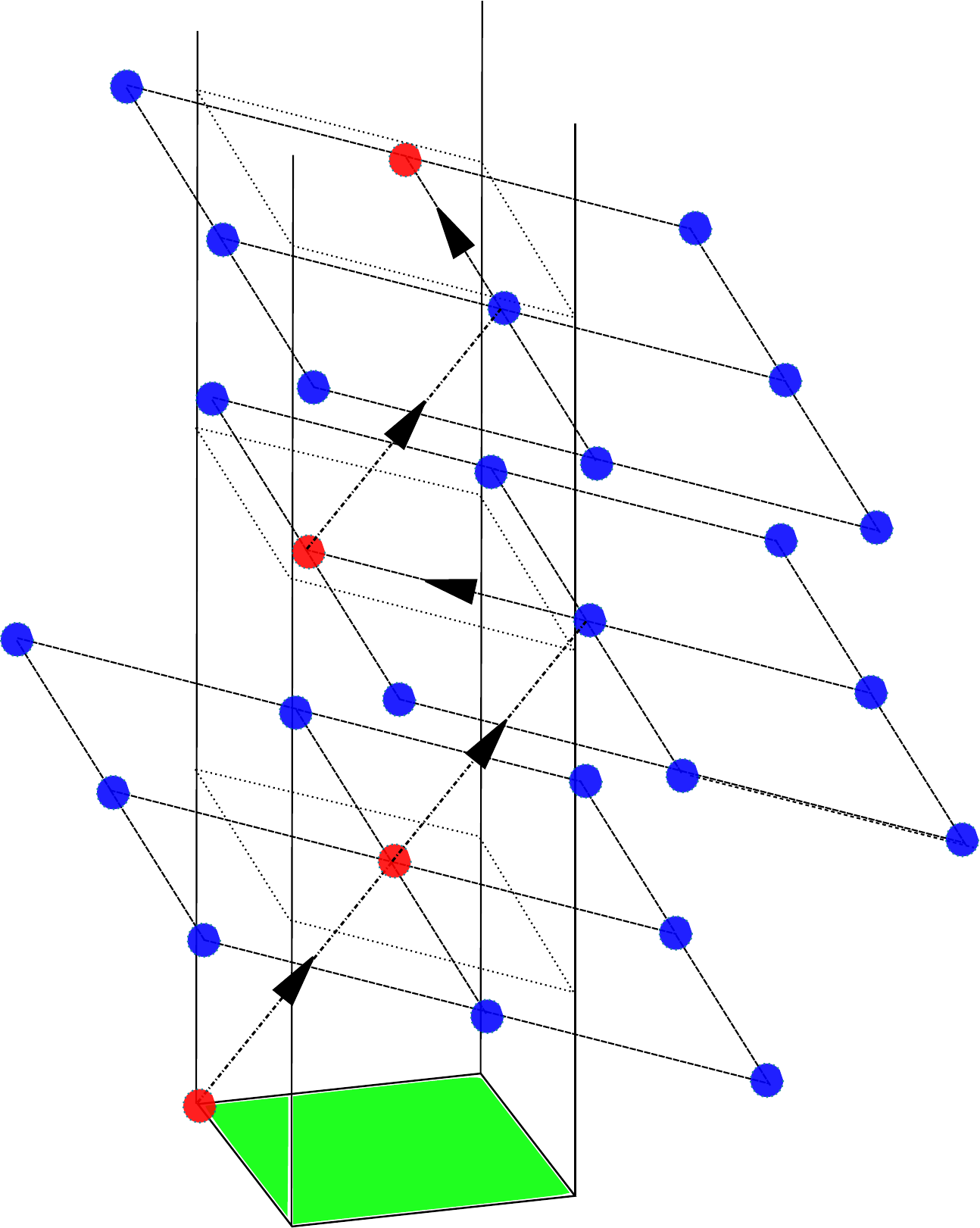}
\vspace*{.1in}
\end{figure}

Now for each $k$ let $x_k$ be the unique element of $(\Lambda\cap H_k)\cap\Ps^{-1}(A).$ If $x_k+v_{s+1}$ lies in $\Ps^{-1}(A)$, then it must be $x_{k+1}.$ In any case (see Figure 1), it follows from (S2) that there is a subset $I_k\subseteq\{1,\ldots ,s\}$ with the property that
\[\Ps(x_k)+\Ps(v_{s+1})-\sum_{i\in I_k}\Ps(v_i)\in A,\]
and we therefore have that
\begin{equation}\label{eqn.x^(k)->x^(k+1)}
x_{k+1}=x_k+v_{s+1}-\sum_{i\in I_k}v_i.
\end{equation}
Then by (S4) we have that
\[\Ps_{s+1}(x_{k+1})>\Ps_{s+1}(x_k),\]
and this implies that, for all $k$,
\[S(\Ps(x_k))=\Ps(x_{k+1})=\Ps(x_k)+\beta\mo M.\]
This verifies condition (R2) for all $x\in A\cap(\alpha\Z+\Z^s)$. To finish the proof simply observe that, for any $x\in A,$ we can find a sequence $\{(n_i,a^{(i)})\}_{i=1}^\infty\subseteq \Z\times\Z^s$ such that
\[n_i\alpha-a^{(i)}\rar x\quad\text{and}\quad S(n_i\alpha-a^{(i)})\rar S(x),~\text{as}~i\rar\infty.\]
From this it follows that $S(x)=x+\beta\mo M$.
\end{proof}
Now we explain a method for constructing infinitely many special regions. Denote the positive cone generated by a collection of points $x_1,\ldots ,x_m\in\R^s$ by
\[C^+(x_1,\ldots ,x_m):=\left\{t_1x_1+\cdots +t_mx_m:t_1,\ldots,t_m\ge 0\right\}.\]
Suppose that $\alpha\in\R^s$ is totally irrational, and without loss of generality suppose that $\alpha$ is chosen so that it lies in the cube $[0,1)^s$. Begin with the obvious basis for $\Lambda,$ obtained by taking $v_i=e_i$ for $1\le i\le s$, and $v_{s+1}=\alpha'$. The region $A=\Ps (P(v_1,\ldots v_s))$ is all of $\T^s$, and it is clearly a special region satisfying (S1)-(S4) above.

Now consider what happens when we replace one of the vectors $v_j$, for some $1\le j\le s$, with the vector $v_j':=v_j-v_{s+1}$. The new collection of vectors will still be a fundamental domain for $\Lambda$, and we claim that
\begin{equation}\label{eqn.PosConeCondition}
\Ps(v_{s+1})\in C^+(\Ps(v_1),\ldots ,\Ps(v_j'),\ldots ,\Ps(v_s)).
\end{equation}
To verify this, note that by (S2) we can write
\[\Ps(v_{s+1})=\sum_{i=1}^st_i\Ps(v_i),\]
with $0<t_i<1$ for each $i$. Then we have that
\[\Ps(v_{s+1})=\left(\frac{t_j}{1-t_j}\right)\Ps(v_j')+\sum_{\substack{i=1\\i\not=j}}^s\left(t_i+\frac{t_jt_i}{1-t_j}\right)\Ps(v_i),\]
and since all of the coefficients are positive, we have established (\ref{eqn.PosConeCondition}).

It follows that we can choose {\em non-negative} integers $b_1,\ldots ,b_s$ with the property that the vector
\[v_{s+1}':=v_{s+1}-b_jv_j'-\sum_{\substack{i=1\\i\not=j}}^sb_iv_i\]
satisfies
\[\Ps(v_{s+1}')\in A':=\Ps(P(v_1,\ldots ,v_j',\ldots ,v_s)).\]
In order to verify (S4) for our new region, notice that
\[\Ps_{s+1}(v_1),\ldots ,\Ps_{s+1}(v_s)\le 0\quad \text{ while }\quad \Ps_{s+1}(v_{s+1})>0,\]
which implies that
\[\Ps_{s+1}(v_{s+1}')\ge \Ps_{s+1}(v_{s+1})\quad\text{ and }\quad\Ps_{s+1}(v_j')< \Ps_{s+1}(v_j).\]
It therefore follows that $A'$ together with $v_1,\ldots ,v_j',\ldots v_s,v_{s+1}'$ satisfy conditions (S1)-(S4) above. By iteratively relabelling and repeating this argument, we can construct as many new examples of non-trivial special regions as we wish.

For comparison, the reader may wish to note that, in the case $s=1$ the procedure we have described here is exactly analogous to the simple continued fraction algorithm.

\section{Proof of Theorem \ref{thm.main}}\label{sec.SpecRegToSpecSec}
As in the introduction, by re-parameterizing in a sub-torus if necessary, we may assume throughout this section that $V$ is a totally irrational $d$-dimensional subspace of $\R^k$. We parameterize $V$ by choosing real numbers $\alpha_i^{(j)}$, for $1\le i\le d$ and $1\le j\le k-d$, such that
\[
V=\left\{\left(y_1,\dots, y_d, \sum_{i=1}^d\alpha_i^{(1)}y_i, \dots, \sum_{i=1}^d\alpha_i^{(k-d)}y_i\right):y_1,\ldots ,y_d\in\R\right\}.
\]
Let $s=k-d$ and set $\alpha= (\alpha_1^{(1)},\dots,\alpha_{1}^{(s)})\in\R^s$. Since $V$ is a totally irrational subspace of $\R^k$, it follows that $\alpha$ is totally irrational in $\R^s$.

Let $A$ be any special region for the irrational rotation of $\R^s$ by $\alpha$, and let $\mc{S}$ be the canonical embedding of $A$ into the subspace generated by $e_{d+1},\ldots ,e_{k}$ in $\R^k$. We will refer to any section $\mc{S}$ constructed in this way as a {\em special section} for $V$, and we may identify $\mc{S}$ with $A$ when there is no ambiguity in doing so.
\begin{proof}[Proof of Theorem \ref{thm.main}]
We wish to show that, for any special section $\mc{S}$, and for any $x\in\R^k$, the set $Y_{\mc{S},x}$ is BD to a lattice. An equivalent problem is to show that, for any set $\mc{S}$ which is a translate of a special section by an element of the subspace generated by $e_{d+1},\ldots ,e_{k}$, the set $Y_{\mc{S},0}$ is BD to a lattice. Therefore suppose that $\mc{S}$ is a section of the latter form.

Using the notation described above, observe that
\begin{align*}
Y_{\mc{S},0}=\bigg\{\Big(n_1,\dots, n_d, &\sum_{i=1}^d\alpha_i^{(1)}n_i, \dots, \sum_{i=1}^d\alpha_i^{(k-d)}n_i\Big):n_1,\ldots ,n_d\in\Z,\\
&n_1\alpha\in A-\gamma(n_2,\ldots ,n_d)\bmod\Z^s\bigg\},
\end{align*}
where $A$ is a special region for $\alpha,$ and $\gamma(n_2,\ldots ,n_d)\in\R^s$. There is a linear map from $Y$ to the set $Y'\subseteq\R^s$ defined by
\begin{align*}
Y'=\left\{(n_1,\dots, n_d)\in\Z^s: n_1\alpha\in A-\gamma(n_2,\ldots ,n_d)\bmod\Z^s\right\},
\end{align*}
therefore it is sufficient for us to show that $Y'$ is BD to a lattice. We remark that this part of the argument necessarily introduces an additional rescaling constant, which depends only on $V$, in the final BD map for $Y$ itself.

For each $(d-1)-$tuple of integers $(n_2,\ldots ,n_d)$ write
\[\left\{n_1\in\Z:n_1\alpha\in A-\gamma(n_2,\ldots ,n_d)\bmod\Z^s\right\}=\{\ell_i(n_2,\ldots ,n_d)\}_{i\in\Z},\]
with $\ell_i<\ell_{i+1}$ and $\ell_{-1}<0\le \ell_0$. Consider the map from $Y'$ to
\[|A|^{-1}\Z\times\Z^{d-1}\]
defined by
\[(\ell_i(n_2,\ldots ,n_d),n_2,\ldots ,n_d)\mapsto \left(i|A|^{-1},n_2,\ldots ,n_d\right).\]

By Theorem \ref{thm.mainBRS}, there exists a constant $C$, which only depends on $A$, such that for any $n_2,\ldots ,n_d$ and for any $i\in\N,$
\[
\left|\ell_i(n_2,\ldots ,n_{d})-i|A|^{-1}\right|\le C|A|^{-1}.
\]
It is easy to see from the definition of a BRS (e.g. see the comment at the end of the proof of \cite[Theorem 3.6]{Hayn2013}) that this inequality also holds for all $i\le 0$, and this proves that the map defined above is a BD map from $Y'$ to a lattice.
\end{proof}

\end{document}